\documentclass[a4paper,12pt,oneside,reqno]{amsart}

\usepackage{amssymb,amsfonts,amsmath,amsthm,amscd, bbm}
\usepackage{graphicx, color}
\numberwithin{equation}{section}  

\newcommand{\beq}{\begin{equation}} 
\newcommand{\eeq}{\end{equation}} 
\newcommand{\bea}{\begin{aligned}}
\newcommand{\eea}{\end{aligned}}
\newcommand{\bdm}{\begin{displaymath}}
\newcommand{\edm}{\end{displaymath}}
\newcommand{\barr}{\begin{array}}
\newcommand{\earr}{\end{array}}
\newcommand{\ben}{\begin{enumerate}}
\newcommand{\een}{\end{enumerate}}
\newcommand{\bde}{\begin{description}}
\newcommand{\ede}{\end{description}}

\newtheorem{thm}{Theorem}
\newtheorem{prop}[thm]{Proposition}  
\newtheorem{lem}[thm]{Lemma}

\newtheorem{rem}[thm]{Remark}
\newcommand{\R}{\mathbb{R}}

\newcommand{\N}{\mathbb{N}}

\newcommand{\PP}{\mathbb{P}}
\newcommand{\E}{{\mathbb{E}}}

\newcommand{\defi}{\equiv} 

\newcommand{\dd}{\text{d}}
\newcommand{\ee}{\text{e}}

\newcommand{\w}{\omega}

\newcommand{\vare}{\varepsilon}

\newcommand{\F}{\mathcal{F}}
\newcommand{\1}{\mathbbm{1}}

\newcommand{\FF}{{\mathcal F}}

\begin{document}
\title[Ergodic theorem for branching Brownian motion]
{An ergodic theorem for the extremal process \\ of branching Brownian motion}

\author[L.-P. Arguin]{Louis-Pierre  Arguin}            
 \address{L.-P. Arguin\\ Universit\'{e} de Montr\'{e}al\\ 2920 chemin de la Tour\\
Montr\'{e}al, QC H3T 1J4 \\ 
Canada}
\email{arguinlp@dms.umontreal.ca}
\author[A. Bovier]{Anton Bovier}
\address{A. Bovier\\Institut f\"ur Angewandte Mathematik\\Rheinische
   Friedrich-Wilhelms-Uni\-ver\-si\-t\"at Bonn\\Endenicher Allee 60\\ 53115
   Bonn,Germany}
\email{bovier@uni-bonn.de}

\author[N. Kistler]{Nicola Kistler}
\address{N. Kistler\\Institut f\"ur Angewandte Mathematik\\Rheinische
   Friedrich-Wilhelms-Uni\-ver\-si\-t\"at Bonn\\Endenicher Allee 60\\ 53115
   Bonn,
Germany}
\email{nkistler@uni-bonn.de}


\thanks{A. Bovier is partially supported through the German Research Foundation in the SFB 611 and
the Hausdorff Center for Mathematics. N. Kistler is partially supported by the Hausdorff Center for Mathematics.
L.-P. Arguin is supported by a Discovery Grant of the Natural Sciences and Engineering Research Council of Canada and the {\it \'Etablissement de Nouveaux Chercheurs} program  from the Fonds de recherche du Qu\'ebec - Nature et technologies.}

 \date{\today}

\begin{abstract} 
In a previous paper, the authors proved a conjecture of Lalley and Sellke that the 
empirical (time-averaged) distribution function of the maximum of branching Brownian motion converges almost surely to a Gumbel distribution.
The result is extended here to the entire system of particles that are extremal, i.e.~ close to the maximum. 
Namely, it is proved that the distribution of extremal particles under time-average converges to a Poisson cluster process.
\end{abstract}

\maketitle

\section{Introduction and Main Result}
Let $x(t)=(x_v(t), v\in \Sigma(t))$ be a standard branching Brownian motion (BBM) on $\R$ defined on a filtered space 
$(\Omega,\FF,\PP,\{\FF_t\}_{t\in\R_+})$.
The set $\Sigma(t)$ indexes the particles at time $t$ and  $x_v(t)$ is the position of the {\it particle} $v$ at time $t$.
We recall the construction of the process: at time $0$ we start with a single standard Brownian motion $x_1(t)$ that splits after
an exponential random time $T$ of mean $1$ into $k$ particles
with probability $p_k$, where $\sum_{k=1}^\infty p_k = 1$, $\sum_{k=1}^\infty k p_k = 2$, and $\sum_{k} k(k-1) p_k < \infty$. 
The positions of the $k$ particles are independent Brownian motions starting at $x_1(T)$.
The $k$ particles branch independently and with the same law as the first Brownian particle.
At time $t>0$, there will be a random number $n(t)\equiv |\Sigma(t)|$ of particles  located at $x(t)=(x_v(t), v\in\Sigma(t))$. 
Note that $\E[n(t)]=e^t$.

A fundamental link between the maximum of BBM and partial differential equations 
was observed by McKean \cite{mckean}. 
If $\phi: \R\to\R$ is such that $0\leq \phi \leq 1$, then the function
\beq 
\label{eqn: bbm_repr}
u(t, x) \equiv 1- \E\left[\prod_{v\in\Sigma(t)} \phi(x+x_v(t)) \right]
\eeq
solves the Kolmogorov-Petrovsky-Piscounov equation [KPP],  also referred to as the Fisher-KPP equation,
\beq  \label{eqn: kpp}
 u_t = \frac{1}{2} u_{xx} + (1-u) - \sum_{k=1}^\infty p_k (1-u)^k, 
\eeq 
with initial condition $u(0,x)=1-\phi(x)$. For the case $\phi(x)=\1_{[0,\infty)}(x)$,
$
1-u(t,x)=\PP\left(\max_{v\in\Sigma(t)} x_v(t) \leq x \right)
$
 is the distribution function of the maximum of BBM. 

Results of Kolmogorov, Petrovsky, and Piscounov \cite{kpp} and of Bramson \cite{bramson} established
the convergence of the distribution under appropriate recentering. 
Namely, for the initial condition  $\phi(x)=\1_{[0,\infty)}(x)$,
\beq \label{travelling_one}
u\big(t, m(t)+ x \big) = 1-w(x) \qquad \text{uniformly in}\;  x\; \text{as} \; t\to \infty,
\eeq
with the recentering term
\beq 
\label{eqn: recentering}
m(t) = \sqrt{2} t - \frac{3}{2\sqrt{2}} \log t, 
\eeq
and $\w(x)$ is the unique solution (up to translation) of a certain ode. 
Convergence for other initial conditions was also proved by Bramson in \cite{bramson_monograph}; see Theorem \ref{bramson_fundamental_convergence} in the Appendix for a precise statement.
A probabilistic interpretation of  $w(x)$ of BBM was given  by Lalley and Sellke.
Define the martingales
\beq \label{defi_martingale}
Y(t) \equiv \sum_{v\in\Sigma(t)} e^{-\sqrt{2} (\sqrt{2}t -x_v(t))} \qquad 
Z(t) \equiv \sum_{v\in\Sigma(t)} \big(\sqrt{2}t -x_v(t)\big) e^{-\sqrt{2} (\sqrt{2}t -x_v(t))}.
\eeq
Then $Y(t)$ converges to zero almost surely while $Z(t)$,
known as  the \emph{derivative martingale}, converges almost random variable $Z>0$.
Moreover, 
\beq
\label{eqn: rep}
w(x)=\E\left[ \exp \left(-C_{\max} Z~e^{-\sqrt{2}x}\right)\right],
\eeq
for an explicitly known  constant $C_{\max}>0$.

In this paper we study properties of the so-called \emph{extremal process} of BBM, i.e.~the point process 
\beq
\mathcal E_{t,\omega}=\sum_{v\in\Sigma(t)}\delta_{x_v(t)-m(t)}.
\eeq
Our objective is to prove weak convergence of  the random point measure $\mathcal E_{t,\omega}$ with respect to {\it time-averaging} for a fixed realization $\omega$.

\begin{rem} The natural topology for point measures is that of \emph{vague convergence}. Weak convergence of random elements of this space is implied by the convergence of Laplace functionals,
\begin{equation}
\E\left[ \exp \left( -\int f(y) ~\mu_n(\dd y) \right) \right]\
\end{equation}
for every $f\in \mathcal C^+_c(\R)$ i.e. the set of non-negative continuous functions with compact support,
see e.g. \cite{kallenberg}.
\end{rem}

To this aim, for all $f\in\mathcal C_c(\R)^+$, we analyze the convergence of the Laplace functional
\begin{equation}
\left\langle~ \exp \left( -\int f(y) ~\mathcal E_{t,\omega}(\dd y) \right) ~\right\rangle_T  \equiv \frac{1}{T}\int_0^T   \exp \left( -\int f(y) ~\mathcal E_{t,\omega}(\dd y) \right)  \dd t \ .
\end{equation}
The limit point process, denoted $\mathcal E_{Z(\omega)}$, is described precisely in \eqref{eqn: E_Z}  after the statement of the result.  
It is a Poisson cluster process whose law depends on the realization $\omega$ through the value of the derivative martingale. 
We write $E$ for the expectation of this point process for a fixed $Z(\omega)$. 
The main result proves in effect an ergodic theorem for the system of extremal particles of BBM. 
However, the system has more than one ergodic components since 
 the limit distribution of the particles do depend on the realization $\omega$ through $Z(\omega)$. 

\begin{thm}[Ergodic theorem for the extremal process]
\label{thm: ergodic}
There exists a set $\Omega_0\subset \Omega$ of $\PP$-probability one on which $\mathcal E_{t,\omega}$ converges weakly under time-average to a Poisson cluster process $\mathcal E_{Z(\omega)}$, where 
$Z(\omega)$ is the limit of the derivative martingale.
That is, for $\omega\in\Omega_0$,
\beq
\label{eqn: ergodic}
\lim_{T\to\infty}\left\langle~ \exp \left( -\int f(y) ~\mathcal E_{t,\omega}(\dd y) \right) ~\right\rangle_T = E\left[  \exp \left(-\int f(y) ~\mathcal E_{Z(\omega)}(\dd y) \right) \right]\ \text{$\forall f\in \mathcal C^+_c(\R)$ .}
\eeq
\end{thm}

The main result has to be compared with the convergence in {\it space-average} of the extremal process $\mathcal E_{t,\omega}$.
From this perspective, one considers the law of  $\mathcal E_{t,\omega}$ under $\PP$ when averaging over the realizations $\omega$
instead of under $\langle ~ \cdot ~ \rangle_T$ for $\omega$ fixed. 
The weak convergence of the extremal process under space-average has been studied
by several authors: Brunet and Derrida \cite{derrida_brunet, brunet_derrida_two},  A\"\i dekon, Berestycki, Brunet, and Shi \cite{aidekon_et_al},
and the present authors  in \cite{abk_poissonian, abk_extremal}. 
A description of the extremal process in the limit has been proved independently in \cite{abk_extremal} and \cite{aidekon_et_al}.
In \cite{abk_extremal}, the existence of the following process needed for the description of the limit is proved.
\begin{thm}[\cite{abk_extremal}]
\label{thm: cluster}
Under the law $\PP$ conditioned on the event $\{ \max_{v\in\Sigma(t)} x_v(t)-\sqrt{2}t  >0\}$,
the point process
\beq
 \overline{\mathcal{E}}_t\defi \sum_{v\in\Sigma(t)} \delta_{x_v(t)-\sqrt{2}t}
\eeq
 converges weakly as $t\to\infty$ to a well-defined point process $ \overline{\mathcal{E}}$.
\end{thm}

If we write  $\overline{\mathcal{E}}= \sum_{k\in\N}\delta_{y_k}$, then the process of the gaps given by
\begin{equation}
\label{eqn: D}
\mathcal D\equiv \sum_{k \in \N} \delta_{y_k - \max_k y_k},
\end{equation}
also exists. This process is referred to as {\it the cluster process}.

Now, let $\mathcal E_Z$ be the Poisson cluster process constructed as follows: 
for $Z>0$ fixed, let $(p_i,i\in\N)$ be a Poisson random measure with intensity $C_{\max} Z \sqrt{2} e^{-\sqrt{2} x} \dd x $ where $C_{\max}$ is the constant appearing in \eqref{eqn: rep}. 
Let $\mathcal{D}^{(i)}=(\Delta_{j}^{(i)}, j\in\N)$ be i.i.d. copies of the cluster point process $\mathcal{D}$ defined in \eqref{eqn: D}. 
Then, for a given $Z>0$, take, 
\begin{equation}
\label{eqn: E_Z}
\mathcal E_Z = \sum_{i,j\in\N} \delta_{p_i + \Delta_{j}^{(i)}}.
\end{equation}

The convergence when $t\to\infty$ under $\PP$ of the random measure $\mathcal E_{t,\omega}$ was proved in \cite{abk_extremal} and in \cite{aidekon_et_al}.
The difference with Theorem \ref{thm: ergodic} is that the dependence on $Z(\omega)$ is  averaged.
\begin{thm}[\cite{abk_extremal}]
\label{thm: extremal}
The random measure $\mathcal E_{t,\cdot}$ converges weakly under $\PP$ and as $t\to\infty$ to a mixture of Poisson cluster processes.
More precisely,  for every $f\in \mathcal C^+_c(\R)$, 
\begin{equation}
\label{eqn: space}
\lim_{t\uparrow\infty}\E\left[ \exp \left( -\int f(y) \mathcal E_{t,\omega}(\dd y) \right) \right] =
\E\left[  E\left[\exp \left(-\int f(y) \mathcal E_{Z}(\dd y) \right)\right] \right].
\end{equation}
(In the right side of \eqref{eqn: space}, the expectation $E$ is over the point process $\mathcal E_Z$ defined in \eqref{eqn: E_Z} for $Z$ fixed, whereas the expectation $\E$
is over the random variable $Z$).
\end{thm}

The proof of Theorem \ref{thm: ergodic} goes along the line of the proof of the convergence of the law of the maximum under time-average to a Gumbel distribution proved in \cite{abk_ergodic} and first conjectured in \cite{lalley_sellke}.
\begin{thm}[\cite{abk_ergodic}]
\label{thm: ergodic max}
The random variable $\max_{v\in\Sigma(t)} x_v(t)-m(t)$ converges weakly under $\langle \cdot \rangle_T$
to a Gumbel distribution for $\PP$-almost all $\omega$. Precisely, for $\PP$-almost all $\omega$, 
\beq
\lim_{T\uparrow \infty}\langle \1_{\{\max_{v\in\Sigma(t)} x_v(t)-m(t) \leq x\}} \rangle_T = \exp \Big(-C Z(\omega) e^{-\sqrt{2}x}\Big), \text{ for all $x\in \R$.}
\eeq
\end{thm}
Theorem \ref{thm: ergodic} extends this result.
In particular, a precise result on the branching times of extremal particles at different time scales is needed, cf. Theorem \ref{thm: overlaps}. This result
 is of independent interest.


\section{Outline of the Proof}

To prove Theorem \ref{thm: ergodic}, one has to find $\Omega_0$ of probability one on which
\beq
\left\langle~ \exp \left( -\int f(y) ~\mathcal E_{t,\omega}(\dd y) \right) ~\right\rangle_T
\eeq
converges simultaneously for all $f\in \mathcal C_c^+(\R)$. As explained in Section \ref{sect: approx}, the convergence on countable set of functions in $ \mathcal C_c^+(\R)$ is in fact  sufficient. Thus one only needs to prove almost sure convergence for a given function $f$. Moreover, due to the fact that the recentered maximum of BBM is stochastically bounded, one can introduce a cutoff on large values of $y$.

Take $\varepsilon>0$ and $R_T$ such that $R_T \ll T$. 
For a given $f\in \mathcal C_c^+(\R)$, consider the decomposition
\begin{eqnarray}
\left\langle~ \exp \left( -\int f(y) \mathcal E_{t,\omega}(\dd y) \right) \right\rangle_T 
& = &\frac{1}{T}\int_0^{\varepsilon T}  \exp \left( -\int f(y)  \mathcal E_{t,\omega}(\dd y) \right) ~ \dd t \label{eqn: decomp 1}\\
&+ &
\frac{1}{T}\int_{\varepsilon T}^{ T}  \E\left[ \exp \left( -\int f(y) \mathcal E_{t,\omega}(\dd y) \right) \Big|  \F_{R_T}\right] \dd t \label{eqn: decomp 2}\\
&+&
\frac{1}{T}\int_{\varepsilon T}^T Y_t(\omega)  \dd t, \label{eqn: decomp 3}
\end{eqnarray}
where 
\begin{equation}
\label{eqn: Y}
Y_t(\omega)\equiv  \exp \left( -\int f(y) ~\mathcal E_{t,\omega}(\dd y) \right) - \E\left[\exp \left( -\int f(y) \mathcal E_{t,\omega}(\dd y) \right) \Big|  \F_{R_T}\right].
\end{equation}
The term \eqref{eqn: decomp 1} can be made arbitrarily small uniformly in $T$ by taking $\varepsilon$ small.
The term \eqref{eqn: decomp 2} is shown to converge almost surely to the right side of \eqref{eqn: ergodic} in Section \ref{sect: LS}.
The treatment is based on the convergence result of Lalley and Sellke \cite{lalley_sellke} and is a generalization of Theorem 2 in \cite{abk_ergodic}.
The condition $t\in [\varepsilon T, T]$ is needed there, because one needs $R_T\ll t$. 
The term \eqref{eqn: decomp 3} is shown to converge to $0$ almost surely in Section \ref{sect: LLN}. 
This is similar in spirit to the law of large numbers proved in \cite{abk_ergodic}. 
However, the proof is simplified here by the explicit use of a new theorem of independent interest about the common ancestor of extremal particles at two different times.
Precisely, for a fixed compact interval $I$, denote by $\Sigma_I(t)$ the particles at time $t$ that are in the set $I+m(t)$,
\begin{equation}
\label{eqn: Sigma}
\Sigma_I(t)\equiv\{v\in \Sigma(t): x_v(t)- m(t)\in I\}.
\end{equation}
Take $t>0$ and $t'>0$ such that $t'>t$. Let $v\in \Sigma_I(t)$ and $v'\in \Sigma_I(t')$. 
Define the {\it branching time} between $v$ and $v'$ as
\begin{equation}
\label{eqn: Q}
\text{ for $v \in  \Sigma_I(t)$ and  $v'\in \Sigma_I(t')$}, ~Q(v,v')=\sup\{ s\leq t: x_v(s)=x_{v'}(s)\}.
\end{equation}
Note that if $v'$ is a descendant of $v$ then $Q(v,v')=t$. 
\begin{thm}
\label{thm: overlaps}
Let $I$ be a compact interval of $\R$. 
There exist $C>0$ and $\kappa>0$ such that
\begin{equation}
\sup_{\substack{t>3r\\ t'>t+r}}\PP\left(\exists v \in \Sigma_I(t), v'\in \Sigma_I(t'): Q(v,v')\in [r,t] \right)\leq C e^{-r^\kappa}.
\end{equation}
\end{thm}
This is a generalization of Theorem 2.1 in  \cite{abk_genealogies}. The technique of proof is similar and is based on the localization of the paths of extremal particles.
The theorem means that with large probability and for two (sufficiently distant) times, the extremal particles come from different ancestors at time $r$. 
Hence, they are conditionally independent given $\F_{r}$. 
This gives enough independence between the variable $Y_s$ to derive the desired convergence of \eqref{eqn: decomp 3} using standard arguments.


\section{Proofs}
\subsection{Approximation in $\mathcal C_c^+(\R)$}
\label{sect: approx}
We first state a lemma that that shows that our task is reduced to prove almost sure convergence for a given function $f$ with an additional cutoff.
\begin{lem}
\label{lem: approx1} Theorem  \ref{thm: ergodic} holds, if 
for any given $f\in \mathcal C^+_c(\R)$, and for any $\delta\in\R$,
\begin{equation}
\lim_{T\to\infty}\left\langle~ \exp \left(-\int f_\delta(y) ~\mathcal E_{t,\omega}(\dd y)\right)~\right\rangle_T = 
E\left[  \exp \left( -\int f_\delta(y) ~\mathcal E_{Z(\omega)}(\dd y) \right) \right]  \ \text{ $\PP$-a.s.,}
\end{equation}
where $f_\delta $ is defined by $\exp(-f_\delta(y))=\exp(-f(y))\1_{(-\infty,\delta]}(y)$.

\end{lem}
\begin{proof}
The Stone-Weierstrass theorem implies the existence of a countable set of dense functions in $\mathcal C_c^+(\R)$ under the uniform topology. 
This reduces the proof to almost sure convergence for a given $f\in\mathcal C_c^+(\R)$.
There is also no loss of generalities in introducing the cutoff $\1_{(-\infty,\delta]}(y)$ since the support of 
$\mathcal E_{\omega,t} $ is stochastically bounded from above (by Theorem \ref{thm: ergodic max}).
\end{proof}


\subsection{A convergence result of Lalley and Sellke}
\label{sect: LS}

For a given $f\in \mathcal C^+_c(\R)$ and $\delta\in\R$, define
\begin{equation}
\label{eqn: u}
u_{f,\delta}(t,x) =1-\E\left[ \prod_{v\in\Sigma(t)} \exp\Big(-f_\delta\big(-x + x_v(t) \big)\Big) \right].
\end{equation}
Note that $u_{f,\delta}(0,x)$ is $0$ for $x$ large enough because of the indicator function.
Convergence of $u_{f,\delta}(t,x+m(t))$ as $t\to\infty$ has been established by Bramson \cite{bramson_monograph}, see Theorem \ref{bramson_fundamental_convergence} in the Appendix for a complete statement.
Using the representation of Lalley and Sellke and arguments of Chauvin and 
Rouault, one can show that (see Lemma 3.8 in \cite{abk_extremal}),
\beq
\label{eqn: conv1}
\lim_{t\to\infty} u_{f,\delta}(t,x+m(t))= \E\left[\exp\Big(-C(f,\delta) Z(\omega) \Big) \right]
\eeq
where
\begin{equation}
\label{eqn: C}
C_r(f,\delta) = \sqrt{\frac{2}{\pi}}  \int_0^\infty u_{f,\delta}(r , y +\sqrt{2} r) ~y e^{y\sqrt{2}}\dd y 
\eeq
and $
C(f,\delta) \equiv \lim_{r\to\infty} C_r(f,\delta)$.
Moreover, it was established in \cite{abk_extremal} (see proof of Theorem 2.6 in \cite{abk_extremal})
 that if $\mathcal D$ is the cluster process with expectation $E$ introduced in \eqref{eqn: D},
then
\begin{eqnarray}
\label{eqn: C cluster}\nonumber
C(f,\delta) 
&=&\int \left(1-E\left[\exp\left(- \int f_\delta(y+z) ~\mathcal D(dz)\right)\right]  \right) ~\sqrt{2}e^{-\sqrt{2}y}\dd y.
\end{eqnarray}
The proof of this in \cite{abk_extremal} is written for without the cutoff but it extends in a straightforward way.
Note that $\exp(-C(f,\delta)Z)$ is  the Laplace functional
of the process $\mathcal E_Z$ for the test function $f$ with the cutoff.

The proof of the next lemma is similar to that of Lemma 4 in 
\cite{abk_ergodic}. 
We present it for completeness.
It is based on an estimate of Bramson, see Proposition 8.3 and its proof in \cite{bramson_monograph}, that were adapted 
to the extremal process setting in Proposition 3.3 of \cite{abk_extremal}.

\begin{lem} 
\label{lem: fundamental_tail}
Consider $t\geq 0$ and $X(t)\geq  0$ such that $\lim_{t\uparrow \infty} X(t) = + \infty$ and
$X(t) = o(\sqrt{t})$. Then, for any fixed $r$ such that $t\geq 8r$ and $t$ large enough so that $X(t)\geq 8r -\frac{3}{2\sqrt{2}}\log(t)$,
\beq \label{eqn: tail_max}
\gamma(r)^{-1} C_r(f,\delta) X(t) e^{-\sqrt{2} X(t)} \big(1+o(1)\big) \leq u_{f,\delta}\big(t, X(t)+m(t)\big) \leq \gamma(r) C_r(f,\delta) X(t) e^{-\sqrt{2} X(t)}
\eeq
where $o(1)$ is a term that tends to $0$ as $t\to\infty$ for $r$ fixed.
\end{lem}

\begin{proof}
Define 
\begin{eqnarray} \label{eqn: ugly_one} \nonumber
&& \psi(r, t, x+\sqrt{2} t) \defi  \frac{e^{-\sqrt{2}x}}{\sqrt{t-r}} \int_0^\infty \frac{ \dd y'}{\sqrt{2\pi}} \cdot u_{f,\delta}(r, y'+\sqrt{2}r) \cdot e^{y'\sqrt{2}}  \\
&& \hspace{2cm} \times  \left\{ 1- \exp\left( -2 y'\frac{x+\frac{3}{2\sqrt{2}}\log t }{t-r} \right) \right\} \exp\left( -\frac{(y'-x)^2}{2(t-r)} \right) .
\end{eqnarray}
Note that
\beq
\int_0^\infty y e^{\sqrt{2} y} u_{f,\delta}(0,y) \dd y < \infty .
\eeq
Proposition 3.3 in \cite{abk_extremal}
implies that for $r$ large enough, $t\geq 8r$, and $x\geq 8r -\frac{3}{2\sqrt{2}}\log t $, 
\beq \label{eqn: ugly_two}
\gamma(r)^{-1}  \psi(r, t, x+\sqrt{2} t) \leq u_{f,\delta}(t, x +\sqrt{2}t) \leq \gamma(r)   \psi(r, t, x+\sqrt{2} t)
\eeq
for some $\gamma(r) \downarrow 1$ as $r\to \infty$. 
As $\sqrt{2} t = m(t)+ \frac{3}{2\sqrt{2}}\log(t)$, by taking $x= X(t) - \frac{3}{2\sqrt{2}}\log t$, and $X=X(t)$
\beq \label{eqn: ugly_three}
\gamma(r)^{-1}  \psi(r, t, X  + m(t)) \leq u_{f,\delta}(t, X + m(t)) \leq \gamma(r)   \psi(r, t, X+ m(t)).
\eeq
It remains to estimate $ \psi(r, t, X  + m(t))$.
By Taylor expansion, since $y'$ is positive and so is $X$, $t$ large enough, 
\beq \label{eqn: ugly_approx} 
\frac{2y'X}{t-r} - \frac{2{y'}^2 X^2}{(t-r)^2}  \leq  1- e^{-2 y'\frac{ X }{t-r} }    \leq \frac{ 2 y' X }{t-r} 
\eeq
and
\beq \label{eqn: ugly_approx.bis} 
1-\frac{(y'- X +\frac{3}{2\sqrt{2}}\log t)^2}{2(t-r)}  \leq  e^{ -\frac{(y'- X +\frac{3}{2\sqrt{2}}\log t)^2}{2(t-r)} }  \leq 1. 
\eeq

The upper bound \eqref{eqn: tail_max} follows from plugging the upper bounds of \eqref{eqn: ugly_approx} and \eqref{eqn: ugly_approx.bis} into 
\eqref{eqn: ugly_one}.
As for the lower bound, note that the lower order terms of the lower bounds \eqref{eqn: ugly_approx} and  \eqref{eqn: ugly_approx.bis}
are all of order $o(1)$, when $t\to\infty$ for $r$ fixed. 
To complete the proof, one has to prove that for $r$ fixed,
\beq
 \sqrt{ \frac{ 2 }{\pi}} \int_0^\infty  (y')^n  e^{y'\sqrt{2}}   u_{f,\delta}(r, y'+\sqrt{2}r)~ \dd y' <\infty, \text{ for $n=0,1,2,3$.}
\eeq
The integrability is shown by bounding $u_{f,\delta}$ above by the solution 
of the linearized KPP equation. 
This is done exactly as in Proposition 3.4, equation (3.18), in 
\cite{abk_extremal} following the argument of \cite{chauvin_rouault}.
We refer the reader to these papers for the details.
\end{proof} 

\begin{prop}
\label{prop: LS}
Fix $\vare >0$. 
Let  $R_T =o(\sqrt{T})$ with $\lim_{T\to\infty} R_T=+\infty$. Then for any $t\in [\vare T, T]$, 
\beq
\label{eqn: bond identity}
\lim_{T \uparrow \infty} 
\E \left[  \exp\left(-\int f_\delta(y) ~\mathcal{E}_{t,\omega }(\dd y)\right) \Big| \F_{R_T}\right]
= \exp\left( - C(f,\delta) Z(\omega)  \right) \text{ $\PP$-a.s.}
\eeq
\end{prop}

\begin{rem}
The right side of \eqref{eqn: bond identity} equals the right side of \eqref{eqn: ergodic} in Theorem \ref{thm: ergodic}.
The connection is through \eqref{eqn: C cluster}.
\end{rem}

\begin{proof}
This is an application of Lemma \ref{lem: fundamental_tail} and the convergence of the derivative martingale.
Enumerate the particles at time $R_T$ by $i=1,\dots, n(R_T)$, and write $x_i(t)$ for the position of the particular $i$.
For a particle $v\in\Sigma(t)$ at time $s$, define $i_v$ to be the index of the ancestor of $v$ at time $R_T$, and $x^{(i_v)}_v(t,R_T) \equiv x_v(t)- x_{i_v}(R_T)$. 
By the Markov property of BBM, conditionally on $\F_{R_T}$, the processes $\big(x^{(i)}_v(t,R_T), v \text{ such that } i_v=i\big)$, $i=1,\dots, n(R_T)$  are independent and distributed 
as the particles of a BBM at time $t-R_T$. We have
\begin{equation}
\label{eqn: product}
\exp \left( -\int f_\delta(y) ~\mathcal E_{t ,\omega}(\dd y)\right)
=\prod_{i=1}^{n(R_T)} \exp \left( -\sum_{{v\in \Sigma(t)}\atop{ i_v=i}}f_\delta(-y_i(R_T) + x^{(i)}_v(t,R_T)-m(t-R_T) )\right)
\end{equation}
where
\beq
y_i(R_T)\equiv\sqrt{2} R_T - x_i(R_T) + \frac{3}{2\sqrt{2}} \log \left(\frac{t-R_T}{t}\right)=\sqrt{2} R_T - x_i(R_T) +o(1),
\eeq
and  $o(1)\downarrow 0$ as $T\uparrow \infty$.
 Using the independence of the $x^{(i)}$'s and the definition \eqref{eqn: conv1}, the conditional expectation of \eqref{eqn: product} given $\F_{R_T}$ can be written as
\beq
\exp \left(\sum_{j \leq n(R_T)} \log \big( 1-u_{f,\delta}(t-R_T,  y_{j}(R_T) +m(t-R_T) \big) \right).
\eeq
Note that the convergence of the martingales defined in 
\eqref{defi_martingale} implies that 
\beq
\lim_{T\uparrow \infty}\min_{j\leq n(R_T)}y_j(R_T)=+\infty.
\eeq
Since for $0<u<1/2$, one has $-u-u^2\leq \log(1-u)\leq -u$,  one can pick $T$ large enough so that
\beq
\E \left[ \exp \left( -\int f_\delta(y) ~\mathcal E_{t ,\omega}(\dd y)\right) 
\Big| \F_{R_T}\right] \leq \exp \left(- \sum_{j \leq n(R_T)} u_{f,\delta}(t-R_T,  y_{j}(R_T) +m(t-R_T)) \right)
\eeq
and
\begin{eqnarray} \nonumber
&&\hspace{-5mm}\E \left[  \exp \left( -\int f_\delta(y) ~\mathcal E_{t ,\omega}(\dd y)\right)
 \Big| \F_{R_T}\right] \geq \exp \left(-\sum_{j \leq n(R_T)} u_{f,\delta}(t-R_T,  y_{j}(R_T) +m(t-R_T))\right) \\
 &&\hspace{4cm}\times \exp \left(-\sum_{j \leq n(R_T)} \big(u_{f,\delta}(t-R_T,  y_{j}(R_T) +m(t-R_T))\big)^2\right).
\end{eqnarray}
To finish the proof, it suffices to show that
\begin{eqnarray}
\lim_{T\to\infty}  \sum_{j \leq n(R_T)} u_{f,\delta}(t-R_T,  y_{j}(R_T) +m(t-R_T)) &=& C(f) Z(\omega)\ , \label{eqn: show1}\\
\lim_{T\to\infty}  \sum_{j \leq n(R_T)} \big(u_{f,\delta}(t-R_T,  y_{j}(R_T) +m(t-R_T))\big)^2&=&0\ . \label{eqn: show2}
\end{eqnarray}
We claim that $y_j(R_T)=o(\sqrt{t})$, uniformly in $j$, so that Lemma \ref{lem: fundamental_tail} can be applied.
Indeed, since $\lim_{T\to\infty} \frac{ \max_{j\leq n(R_T)}x_j(R_T)}{R_T}=\sqrt{2}$, $\PP$-a.s., we have
\beq
\label{eqn: LLN max}
\frac{y_{j}(R_T)}{t^{1/2}} \to 0 \text{ as $T\to\infty$ $\PP$-a.s., uniformly in $j\leq n(R_T)$.}
\eeq
Equation \eqref{eqn: show1}  follows by picking a fixed $r$ in Lemma \ref{lem: fundamental_tail}, taking first 
$T\to\infty$,  then $r\to\infty$, and using the  convergence of the derivative martingale.
Lemma \ref{lem: fundamental_tail} is also used to establish \eqref{eqn: show2}. 
By fixing $r$, the proof is reduced to show that  $ \sum_{j \leq n(R_T)} y_j(R_T)^2e^{-2\sqrt{2} y_j(R_T)}$ goes to zero. 
This is clear since this sum is bounded above by
\beq 
\max_{j\leq n(R_T)}\left( y_j(R_T)^2 e^{-\sqrt{2}y_j(R_T)}\right) \times \sum_{j\leq n(R_T)} e^{-\sqrt{2}y_j(R_T)}
\eeq
and both terms tend to zero almost surely as $T\uparrow\infty$.
\end{proof}


\subsection{Proof of Theorem \ref{thm: overlaps}}
\label{sect: branching}

Throughout the proof, $C$ and $\kappa$ will denote generic constants that do not depend on $t,t'$, and $r$
and that are not necessarily the same at different occurrences.
We recall the result on the localization of the paths of extremal particles established in 
\cite{abk_genealogies}.
Let $t>0$ and $\gamma>0$ and define  
\beq \label{eqn: F}
\begin{aligned}
f_{\gamma,t}(s) 
&\defi \begin{cases}
                    s^\gamma & 0\leq s \leq t/2, \\
		    (t-s)^\gamma & t/2\leq s \leq t. 
                   \end{cases}
                   \\
F_{\alpha,t} (s) 
&\defi \frac{s}{t} m(t) - f_{\alpha,t}(s), \quad 0\leq s\leq t.
\end{aligned}
\eeq
Fix $0 <\alpha<  1/2 < \beta < 1$. By definition, $F_{\beta,t} (s) < F_{\alpha,t} (s)$,  $F_{\beta,t} (0) = F_{\alpha,t}(0) =0$, and $F_{\beta,t} (t) = F_{\alpha,t}(t) =m(t)$.

The following proposition gives strong bounds to the probability of finding
 particles that are close to the level of the maximum at given times but whose paths are 
 bounded by $F_{\beta,t}$ and $F_{\alpha,t}$. (It was stated as Proposition 6 in \cite{abk_ergodic}).
It follows directly from the bounds derived in the course of the proof
of \cite[Corollary 2.6]{abk_genealogies}, cf. equations (5.5), (5.54), (5.62) and (5.63).
\begin{prop}
\label{prop: loc}
Let $I$ be a compact interval. There exist $C>0$ and $\kappa>0$ (depending on 
$\alpha, \beta$ and $I$) such that
\beq \label{tube} 
 \sup_{t\geq 3r} 
\PP\big[ \exists v\in \Sigma_I(t):  x_v(s)\geq  F_{\alpha,t}(s) \text { or } x_v(s)\leq F_{\beta,t}(s)  \text{ for some } s\in (r,t-r) \big] \leq Ce^{r^\kappa} .
\eeq
\end{prop}

We now prove Theorem \ref{thm: overlaps}.
Fix $\alpha$ and $\beta$ as in Proposition \ref{prop: loc}.
Define the set of extremal particles in $I$ at time $t$ that are localized in the interval $(r_1,t-r_1)$
\beq
\Sigma^{loc}_I(t)=\{v\in \Sigma(t): x_v(t)\in I+m(t),  F_{\beta,t}(s) \leq x_v(s)\leq  F_{\alpha,t}(s) ~ \forall s\in (r_1,t-r_1)\}.
\eeq
The parameter $r_1=r_1( r )$ is chosen to be smaller than $r$ (so that $ (r_1,t-r_1)\subset  (r,t-r)$).
The precise choice will be given below.
By Proposition \ref{prop: loc}, to prove Theorem \ref{thm: overlaps}, it suffices to show
\begin{equation}
\label{eqn: to show}
\sup_{\substack{t>3r\\ t'>t+r}}\PP\left(\exists v \in \Sigma^{loc}_I(t), v'\in \Sigma^{loc}_I(t'): Q(v,v')\in [r,t] \right)\leq C e^{-r^\kappa}
\end{equation}
By Markov's inequality the probability is smaller than
\beq
\E\left[ \#\{(v,v'):  v \in \Sigma^{loc}_I(t), v'\in \Sigma^{loc}_I(t'): Q(v,v')\in [r,t]\}\right].
\eeq
This expectation can be evaluated using Sawyer's formula \cite{sawyer}, see also \cite{bramson}. 
Write $x$ for a standard Brownian motion on $\R$, $\mu_s$ for the standard Gaussian measure of variance $s$.
Let also $\Xi^{t_1,t_2}$ be the set of continuous path on $[0,\infty]$ that satisfy $F_{\beta,t}(s) \leq x_v(s)\leq  F_{\alpha,t}(s) ~ \forall s\in (t_1,t_2)$.
The expectation equals
\begin{equation}
\label{eqn: sawyer}
K e^t \int_r^{t} e^{t'-s} \dd s \int_{\R}\mu_s(\dd y) \PP\left(x\in \Xi^{r,t-r_1} \big| x(s)=y\right)\PP\left(x\in \Xi^{s,t'-r_1} \big| x(s)=y\right)
\end{equation}
where $K=\sum_{k=1}^\infty k(k-1)p_k$.
The second probability is estimated exactly the same way as in \cite{abk_genealogies} using Brownian bridge estimates (from equation (4.5) to (4.14)). 
One needs to pick $r_1( r)< r/2$ and $r_1( r) < r^{1-\beta}$.
The result is 
\beq
\PP\left(x\in \Xi^{s,t'-r_1} \big| x(s)=y\right)\leq C r^{1/2} \frac{e^{-(t'-s)} e^{\frac{3}{2}\frac{t'-s}{t'}\log t'} F_{\beta,t'} (s) e^{-\sqrt{2} F_{\alpha,t'}(s)}}{(t'-s)^{3/2}}
\eeq
Since this bound is uniform in $y$, \eqref{eqn: sawyer} is smaller than 
\begin{equation}
\label{eqn: sawyer 2}
C e^t \PP\left(x\in \Xi^{r_1,t-r_1}\right)~r^{1/2}  \int_r^{t'-r}   \frac{ e^{\frac{3}{2}\frac{t'-s}{t'}\log t'} F_{\beta,t'} (s) e^{-\sqrt{2} F_{\alpha,t'}(s)}}{(t' - s)^{3/2}} ~ \dd s .
\end{equation}
Note that the integral is now on $[r,t'-r]\supset[r,t]$. 
The term $e^t \PP\left(x\in \Xi^{r_1,t-r_1}\right)$ is of order of $r_1>r$ (cf. equation (4.17) in \cite{abk_genealogies}). 
It remains to estimate the integral.
The domain of integration is split into $[r, t'/2]$, $[t'/2,t'-t'^{\nu}]$, and $[t'-t'^{\nu},t'-r]$, for a fixed $\nu>0$.
On the first interval, one has that \eqref{eqn: sawyer 2} is smaller than
\begin{equation}
\label{eqn: bound1}
C r^{3/2}  \int_r^{t'/2}  \frac{ e^{\frac{3}{2}\frac{t'-s}{t'}\log t'} s^\beta  e^{-\sqrt{2} s^\alpha}}{(t'-s)^{3/2}} \dd s \leq 
C r^{3/2}  \int_r^{\infty}   s^\beta  e^{-\sqrt{2} s^\alpha} \dd s \leq C r^{3/2}e^{-r^\alpha}. 
\end{equation}
On the second interval, using the change of variable $s \to t' - s$, one gets the upper bound
\begin{equation}
\label{eqn: bound2}
C r^{3/2}  \int_{t'^{\nu}}^{t'/2}  \frac{ e^{\frac{3}{2}\frac{s}{t'}\log t'} s^\beta  e^{-\sqrt{2} s^\alpha}}{s^{3/2}} \dd s \leq 
C r^{3/2} {t'}^{3/4} \int_{t'^\nu}^{t'/2}    s^\beta  e^{-\sqrt{2} s^\alpha} \dd s \leq C r^{3/2} t'^{7/4 }e^{-t'^{\nu\alpha}}. 
\end{equation}
Note that for $t'>t>3r$, $\kappa$ can be chosen for the bound to be of the desired form.
Finally on $[t'-t'^\nu,t'-r]$, the upper bound is 
\begin{equation}
\label{eqn: bound3}
C r^{3/2}  \int_{r}^{t'^\nu}  \frac{ e^{\frac{3}{2}\frac{s}{t'}\log t'} s^\beta  e^{-\sqrt{2} s^\alpha}}{s^{3/2}} \dd s \leq 
C r^{3/2}  \int_{r}^{\infty}    s^\beta  e^{-\sqrt{2} s^\alpha} \dd s \leq C r^{3/2} e^{-r^{\alpha}}. 
\end{equation}
The constants $C$ and $\kappa$ can be picked such that the bounds \eqref{eqn: bound1}, \eqref{eqn: bound2}, and \eqref{eqn: bound3}
have the desired form. This completes the proof of Theorem \ref{thm: overlaps}.


\subsection{The law of large numbers}
\label{sect: LLN}

In this section, we prove that the term \eqref{eqn: decomp 3} goes to zero as $T$ goes to infinity:
\begin{prop}
\label{prop: LLN}
Let $f\in \mathcal C_c(\R)$, $R_T=o(\sqrt{T})$ with $\lim_{T\to\infty}R_T=+\infty$, and $\varepsilon>0$. Consider $Y_t(\omega)$ defined in \eqref{eqn: Y}.
Then 
\begin{equation}
\label{eqn: LLN}
\lim_{T\uparrow\infty}\frac{1}{T}\int_{\varepsilon T}^T Y_t(\omega)  \dd t =0,
\quad
 \PP-a.s.
\end{equation}
\end{prop}
\begin{proof}
The proof is based on a Theorem of Lyons \cite{lyons} that we cite from 
 \cite{abk_ergodic}:
\begin{thm}[\cite{lyons}, Theorem 8 in \cite{abk_ergodic}]
\label{lyons_int}
Consider a process $\{Y_s\}_{s\in \R_+}$ such that $\E[Y_s]=0$ for all $s$. Assume furthermore that the random variables are 
uniformly bounded, say $\sup_s |Y_s|\leq 2$ almost surely. If
\beq \label{sum_corr}
\sum_{T=1}^\infty \frac{1}{T} \E\left[\Big| \frac{1}{T} \int_0^T Y_s  
\dd s \Big|^2\right]<\infty, 
\eeq
then
\beq
 \frac{1}{T} \int_0^T Y_s ~ ds \to 0, \text{ a.s.}
\eeq
\end{thm}
A straightforward decomposition gives
\begin{eqnarray}\nonumber
\sum_{T=1}^\infty \frac{1}{T} \E\left[\Big| \frac{1}{T} \int_{\varepsilon T}^T Y_s ~\dd s \Big|^2\right]
&=& \sum_{T=1}^\infty \frac{1}{T^3}\int_{\varepsilon T}^T\int_{\varepsilon T}^T \E[Y_sY_{s'}] ~ \dd s \dd s' \\\nonumber
&=& \sum_{T=1}^\infty \frac{1}{T^3}\int_{\substack{s,s'\in[\vare T,T] \\ |s-s'|\leq R_T}} \E[Y_sY_{s'}] ~ \dd s \dd s'\\ &+&  
\sum_{T=1}^\infty \frac{1}{T^3}\int_{\substack{s,s'\in[\vare T,T] \\ |s-s'|\geq R_T}} \E[Y_sY_{s'}] ~ \dd s \dd s'. 
\end{eqnarray}
Since $0\leq Y_s \leq 2$ for any $s$, the first term is smaller than a constant times $\sum_{T=1}^\infty \frac{R_T}{T^2}$, which is summable for $R_T=o(\sqrt{T})$. 
Therefore, to prove Proposition \ref{prop: LLN} using Theorem \ref{lyons_int}, it remains to show that the second term is summable. 
This is done in the following lemma.
\end{proof}

\begin{lem}
Let $Y_s$ be as in \eqref{eqn: Y}. Then for $R_T=o(\sqrt{T})$ with $\lim_{T\to\infty}R_T=+\infty$,
\beq
\E[Y_sY_s'] \leq C e^{-R_T^\kappa} \text{ for any $s,s'\in [\varepsilon T,T]$ with $|s-s'|\geq R_T$. }
\eeq
\end{lem}
\begin{proof}

For the given $f\in \mathcal C_c(\R)$, take the compact interval $I=\text{supp} f$.
Recall the definitions of $\Sigma_I(s)$ and $Q(v,v')$ in \eqref{eqn: Sigma} and \eqref{eqn: Q}.
Consider the events
\begin{equation}
\label{eqn: events}
\begin{aligned}
\mathcal A_{s,T}&\equiv \left\{\exists v\in \Sigma_I(s):  x_v(R_T)\geq  F_{\alpha,s}(R_T) \text { or } x_v(R_T)\leq F_{\beta,s}(R_T)   \right\}\\
\mathcal B_{s,s',T}&\equiv \left\{\exists v \in \Sigma_I(s), v'\in \Sigma_I(s'): Q(v,v')\in [R_T,s]\right\} \text{ for $s'>s$.}\\
\mathcal C_{s,s',T}&\equiv \mathcal A_{s,T} \cup \mathcal A_{s',T} \cup \mathcal B_{s,s',T}.
\end{aligned}
\end{equation}
By Proposition \ref{prop: loc} and Theorem \ref{thm: overlaps}, one has that for constants $C>0$ and $\kappa>0$
\begin{equation}
\label{eqn: prob decay}
\PP\left( \mathcal C_{s,s',T} \right)\leq C e^{-R_T^\kappa} \text{ for any $|s-s'|\geq R_T$.}
\end{equation}
Enumerate (in no particular order) the particles at time $R_T$ by $i=1,\dots, n(R_T)$ and write $x_i(t)$ for the position of the particular $i$.
Recall the notation introduced at the beginning of the proof of Proposition \ref{prop: LS}.
Define
\beq
\Phi_{i}(s)\equiv\exp \left( -\sum_{v\in \Sigma_I(s): i_v=i}f_\delta(-y_i(R_T) + x^{(i)}_v(s,R_T)-m(s-R_T) )\right).
\eeq
Note that $\Phi_i(s)$ is non-trivial (that is, not equal to one) if and only if the particle $i$ 
has descendants in the interval $I+m(s)$ at time $s$ (and similarly for the time $s'$).
The crucial step is to notice that on $\mathcal C^c_{s,s',T}$ (on $\mathcal B^c_{s,s',T}$ in fact), for any $i=1,\dots, n(R_T)$,
no two extremal particles in $I$ at time $s$ and at time $s'$ have branching time in $[R_T,s]$.
In particular, two such extremal particles must have two distinct ancestors at time $R_T$. 
This implies that
\beq
\forall i=1,\dots,n(R_T) \text{, $\Phi_i(s)$ and $\Phi_i(s)$ cannot be simultaneously non-trivial on $\mathcal C^c_{s,s',T}$}.
\eeq
Therefore, the following identity holds on $\mathcal C^c_{s,s',T}$
\begin{equation}
\label{eqn: prod2}
\begin{aligned}
\Phi_i(s)\Phi_i(s') = 1-\big(1- \Phi_i(s) \big) - \big( 1-\Phi_i(s') \big).
\end{aligned}
\end{equation}
Putting all this together, one gets
\beq
\begin{aligned}
&\left|\E\left[\prod_{i=1}^{n(R_T)}\Phi_i(s)\Phi_i(s') \right] -  \E\left[ \prod_{i=1}^{n(R_T)}\Big( 1-\big(1- \Phi_i(s) \big) - \big( 1-\Phi_i(s') \big)\Big)\right]\right|\leq 2 \PP(\mathcal C_{s,s',T}).
\end{aligned}
\eeq
This has the right decay by \eqref{eqn: prob decay}. 
By the definition of $u_{f,\delta}(t,x)$ in \eqref{eqn: u}
\beq
u_i(s)\equiv\E[1-\Phi_i(s') | \F_{R_T}]=u_f\big(s-R_T, y_i(R_T)+m(s-R_T),
\eeq
and by the conditional independence property of BBM,
\begin{equation}
\label{eqn: comparison}
 \E\left[ \prod_{i=1}^{n(R_T)}\Big( 1-\big(1- \Phi_i(s) \big) - \big( 1-\Phi_i(s') \big)\Big)\right]
 =\E\left[\prod_{i=1}^{n(R_T)} \Big(1- u_i(s)- u_i(s')\Big)\right].
\end{equation}
It remains to compare the right side of \eqref{eqn: comparison} with the contribution to the correlation coming out from the second term of \eqref{eqn: Y}:
\beq
\E\left[ \prod_{i=1}^{n(R_T)} \E[\Phi_i(s) | \F_{R_T}]  \prod_{i=1}^{n(R_T)} \E[\Phi_i(s') | \F_{R_T}]\right]=
\E\left[ \prod_{i=1}^{n(R_T)}( 1 - u_i(s) )(1- u_i(s'))\right].
\eeq
Again, 
\beq
\begin{aligned}
&0 \leq 
\E\left[ \prod_{i=1}^{n(R_T)}( 1 - u_i(s) )(1- u_i(s')) ~;~ \mathcal C_{s,s',T} \right] -  \E\left[\prod_{i=1}^{n(R_T)} \Big(1- u_i(s)- u_i(s')\Big) ~;~ \mathcal C_{s,s',T}\right]\\
& \hspace{12 cm} \leq 2 \PP\left(\mathcal C_{s,s',T} \right),
\end{aligned}
\eeq
so that by \eqref{eqn: prob decay}, it suffices to bound the difference of the two terms on $\mathcal C^c_{s,s',T}$. 
By the definition of this event in  \eqref{eqn: events}, the only $i$'s whose contribution to the product is not one are those
for which $F_{\alpha,s}(R_T) < x_i(R_T)< F_{\beta,s}(R_T)$ and $F_{\alpha,s'}(R_T) < x_i(R_T)< F_{\beta,s'}(R_T)$
Define
\begin{equation}
\label{eqn: Delta}
\begin{aligned}
&\Delta\equiv \{i=1\dots, n(R_T): F_{\alpha,s}(R_T) < x_i(R_T) < F_{\beta,s}(R_T) \\
&\hspace{8cm} \text{ and } F_{\alpha,s'}(R_T) < x_i(R_T) < F_{\beta,s'}(R_T)\}
\end{aligned}
\end{equation}
By the definition of the functions $F$ in \eqref{eqn: F}, the above condition reduces to $R_T^\alpha+o(1)\leq y_i(R_T)\leq R_T^\beta + o(1)$
where $o(1)$ converges to $0$ as $T\to\infty$ uniformly for $s\in[\varepsilon T, T]$. We are left to bound
\beq
\E\left[ \prod_{i\in \Delta}( 1 - u_i(s) )(1- u_i(s')) - \prod_{i\in\Delta} \Big(1- u_i(s)- u_i(s')\Big)~;~ \mathcal C^c_{s,s',T} \right] \\
\eeq
Since both terms are smaller than one, we can use the 
Lipschitz property $|e^z-e^{z'}|\leq |z-z'|$ for $z,z'\leq 0$ to bound the above by
\begin{equation}
\label{eqn: last bound}
\E\left[ \sum_{i\in \Delta} \log\Big(1 - u_i(s) - u_i(s')+u_i(s)u_i(s') \Big) - \log\Big(1 - u_i(s) - u_i(s')\Big) \right] 
\end{equation}

By Lemma \ref{lem: fundamental_tail}, for a fixed $r$, 
\beq
u_i(s)= u_f\big(s-R_T, y_i(R_T)+m(s-R_T) \leq \gamma(r) C_r(f,\delta) y_i(R_T) e^{-\sqrt{2} y_i(R_T)}.
\eeq
In particular, by the restrictions on $y_i(R_T)$ for $i\in\Delta$, it is possible to choose $T$ large enough such that $u_i(s)\leq 1/4$ and $u_i(s')\leq 1/4$ for every $i\in\Delta$. 
Since for $0<b<1$ and $a>1/2$, $\log(a+b)-\log a\leq 2b$, \eqref{eqn: last bound} is bounded by
\beq
\begin{aligned}
2 \E\left[\sum_{i\in \Delta} u_i(s)u_i(s')\right]\leq C \E\left[\sum_{i\in \Delta}  y_i(R_T)^2 e^{-2\sqrt{2} y_i(R_T)}\right],
\end{aligned}
\eeq
for some constant $C>0$.
The  variables $(x_i(R_T),i\leq n(R_T))$ are Gaussian. We take advantage of the linearity and that $\E[n(R_T)]=e^{R_T}$ to get the upper bound
\beq
C e^{R_T}\int_{R_T^\alpha + o(1)}^{R_T^\beta + o(1)} y^2e^{-2\sqrt{2} y} \frac{e^{-\frac{(\sqrt{2}R_T-y)^2}{2R_T}} \dd y}{\sqrt{2\pi R_T}}\leq CR_T^{2\beta} e^{-\sqrt{2}R_T^\alpha}.
\eeq
This concludes the proof of the lemma.
\end{proof}

\appendix 
\section{Convergence of the KPP equation}
Bramson settled the question of the convergence of the solutions of the KPP equation for generic initial conditions in \cite{bramson_monograph}. 
The recentering is a function of the initial conditions. 
We combine here two of the convergence theorems of \cite{bramson_monograph} as well as the condition on $u(0,x)$ to get the specific form \eqref{eqn: recentering} for the recentering.
\begin{thm}[Theorem A, Theorem B, and Example 2 in \cite{bramson_monograph}] 
\label{bramson_fundamental_convergence}
 Let $u$ be a solution of the F-KPP equation \eqref{eqn: kpp} with $0\leq u(0,x) \leq 1$. Then 
\beq \label{kpp_general_form}
u(t, x+m(t)) \to1- w(x), \quad \text{uniformly in} \; x\; \text{as} \; t\to \infty,
\eeq
where $w$ is the unique solution (up to translation) of
\beq
\frac{1}{2} w'' + \sqrt{2} w'+  \sum_{k=1}^\infty p_k w^k -w = 0,
\eeq
if and only if 
\begin{itemize}
\item[1.] for some $h>0$, $\limsup_{t\to \infty} \frac{1}{t} \log \int_t^{t(1+h)} u(0,y) \dd y \leq -\sqrt{2}$;
\item[2.] and for some $\nu>0$, $M>0$, $N>0$, $\int_{x}^{x+N} u(0,y) \dd y > \nu$ for all $x\leq - M$. 
\end{itemize}
Moreover, if $\lim_{x\to \infty} \ee^{b x}u(0,x) = 0$ for some $b> \sqrt{2}$, then one may choose
\beq \label{choice_m} 
m(t) = \sqrt{2} t -\frac{3}{2\sqrt{2}} \log t.
\eeq
\end{thm}

\end{document}